\documentclass[12pt,reqno]{amsart}
\usepackage{amsmath,amssymb,amsfonts,amscd,latexsym,amsthm,mathrsfs,verbatim}
\usepackage[unicode]{hyperref}
\usepackage{hypbmsec}
\newtheorem{lemma}{Lemma}
\newtheorem{theorem}{Theorem}

\theoremstyle{remark}

\let\wt\widetilde
\renewcommand{\d}{{\mathrm d}}

\newcommand{\ord}{\operatorname{ord}}
\newcommand\cL{\mathcal{L}}
\newcommand\cF{\mathcal{F}}
\newcommand\cW{\mathcal{W}}
\newcommand{\qbinom}[2]{{\genfrac{[}{]}{0pt}{}{#1}{#2}}_q}
\newcommand{\pbinom}[2]{{\genfrac{[}{]}{0pt}{}{#1}{#2}}_p}
\begin{document}

\title{On the irrationality of generalized $q$-logarithm}

\author{Wadim Zudilin}
\address{School of Mathematical and Physical Sciences, The University of Newcastle, Callaghan NSW 2308, AUSTRALIA}
\email{wzudilin@gmail.com}

\date{11 January 2016. \emph{Revision}: 24 January 2015}

\dedicatory{To my good friends and fine colleagues Peter Bundschuh and Keijo V\"a\"an\"anen}

\begin{abstract}
For integer $p$, $|p|>1$, and generic rational $x$ and $z$, we establish the irrationality of the series
$$
\ell_p(x,z)=x\sum_{n=1}^\infty\frac{z^n}{p^n-x}.
$$
It is a symmetric ($\ell_p(x,z)=\ell_p(z,x)$) generalization of the $q$-logarithmic function ($x=1$ and $p=1/q$ where $|q|<1$),
which in turn generalizes the $q$-harmonic series ($x=z=1$). Our proof makes use of the Hankel determinants built on the
Pad\'e approximations to $\ell_p(x,z)$.
\end{abstract}


\thanks{The work is supported by the Australian Research Council.}

\maketitle

\section{Historical notes}
\label{s1}

In 1827, a problem was published about the series
\begin{equation}
q+2q^2+2q^3+3q^4+2q^5+4q^6+2q^7+\dotsb=\sum_{n=1}^\infty\frac{q^n}{1-q^n}
\label{qh1}
\end{equation}
(see \cite[Aufgabe 16, p.~99]{Cr27}). It was pointed out that J.~Lambert had shown that the coefficient of $q^n$ in the series
counts the number of divisors of $n$, so that the coefficient 2 may only appear when the exponent of $q$ is a prime number.
The question asked was to find a ``finite expression'' for the $q$-series. The latter converges inside the disc $|q|<1$ and
is usually referred by the name of $q$-harmonic series, because multiplication of the right-hand side by $1-q$ and the termwise limits
$(1-\nobreak q)\*q^n/(1-q^n)\to1/n$ as $q\to1^-$ formally assign to it the divergent harmonic series
(see \cite[Division~2, Problem~32]{PS78} for a precise limiting behaviour of \eqref{qh1} as $q\to1^-$).
As we are aware these days, the series \eqref{qh1} is quite independent of other known $q$-series, even of its relatives \cite{Pu05},
so it is still an open problem to figure out what other functions can be linearly or algebraically linked to it.
However, as first observed by Th.~Clausen in the final paragraph of \cite{Cl28}, one can easily accelerate the convergence of the $q$-harmonic series:
\begin{equation*}
\sum_{n=1}^\infty\frac{q^n}{1-q^n}
=\sum_{n=1}^\infty q^{n^2}\frac{1+q^n}{1-q^n}
\end{equation*}
(see also \cite[Division~8, Problem~74]{PS76}). This computationally practical formula does not hint at an approach to the arithmetic properties,
like the irrationality and transcendence, of the values of $q$-harmonic series.
Curiously enough, the previous problem (Aufgabe~15) on the list in~\cite{Cr27} requests about the irrationality of the numbers $\pi^n$ for integer $n\ge2$.

The first proof of the irrationality of the $q$-harmonic series for $q$ of the form $q=1/p$, where $p>1$ is an integer,
was recorded by P.~Erd\H os \cite{Er48} in 1948. His argument uses the non-periodicity of base $p$ expansion of the number in question
(and so, certain delicate properties of the number-of-divisors function) and cannot be extended to deal with the more general analytical function
\begin{equation}
L_q(z)=\sum_{n=1}^\infty\frac{q^nz^n}{1-q^n}, \qquad |z|<1, \quad |q|<1,
\label{qh3}
\end{equation}
known as the $q$-logarithm. Independently, J.-P.~B\'ezivin \cite{Be88} in 1988 and P.~Borwein \cite{Bo91} in 1991 pushed forward different analytical methods
to establish the irrationality of the series for rational $z$ and $q$ of the form $q=1/p$ with $p\in\mathbb Z\setminus\{0,\pm1\}$, though B\'ezivin's paper
did not mention $L_q(z)$ specifically\,---\,it was observed only later \cite{BV94} that the general results from \cite{Be88} imply the irrationality.
The advantage of Borwein's method \cite{Bo91,Bo92}, in which he uses the Pad\'e approximations to \eqref{qh3}, is that it allows one to ``measure''
the irrationality of the numbers in question; this quantitative counterpart is absent in Erd\H os' method \cite{Er48} and it was also absent
in B\'ezivin's original method \cite{Be88,Be90} until the recent work~\cite{Ro11} of I.~Rochev (see \cite{RV13} for a further development).
The Pad\'e approximation technique as originated in \cite{Bo91,Bo92} was significantly generalized and extended in later works (for example,
\cite{AZ98,BV94,BV05,BV09,BZ07,BZ08,Du96,KRZ06,Ma08,MVZ06,SvA09,Zu02a,Zu04,Zu06} to list a few) to sharpen the irrationality measures of the values
of $q$-harmonic and $q$-logarithm series as well as to prove the irrationality and linear independence results for some close relatives of the series.

B\'ezivin's method from \cite{Be88,Be90} is capable of dealing with a class of generalized $q$-hypergeometric functions that are quite different from~\eqref{qh3}.
It is a ``lucky accident'' that the $q$-logarithm appears as the quotient of two such functions\,---\,it happens to be the logarithmic derivative
of the so-called exponential function
$$
E_q(z)=\sum_{n=0}^\infty\frac{q^{n(n+1)/2}z^n}{\prod_{j=1}^n(1-q^j)}
=\prod_{n=1}^\infty(1+q^nz);
$$
namely, $L_q(z)=zE_q'(-z)/E_q(-z)$ where the derivative is with respect to~$z$. Though the method from \cite{Be88,Be90} underwent modifications and
generalizations in the later works of B\'ezivin himself and other authors \cite{Be98,Be09,Ch01,KRVZ09,Ro11,RV13}, it essentially serves the original class of functions.

The generalization of the $q$-logarithm we discuss here is given by
\begin{align}
\ell_p(x,z)
&=x\sum_{n=1}^\infty\frac{q^nz^n}{1-q^nx}
=\sum_{m,n>0}q^{mn}x^mz^n
\nonumber\\
&=x\sum_{n=1}^\infty\frac{z^n}{p^n-x}, \qquad |x|<|p| \quad\text{and}\quad |z|<|p|,
\label{e01}
\end{align}
where $q=1/p$ is inside the unit disc. In spite of the several remarkable properties of the function, including its symmetry $\ell_p(x,z)=\ell_p(z,x)$ and
the functional equation
\begin{equation*}
\ell_p(x,z)=\frac{xz}{p-x}+z\ell_p\Bigl(\frac xp,z\Bigr),
\end{equation*}
it has not attracted much attention because of its overall arithmetic toughness. The generalized $q$-logarithm \eqref{e01} was a topic
of our joint investigation \cite{BZ08} with P.~Bundschuh; there we constructed explicitly Pad\'e-type approximations to the function
and applied them to the cases when $x$ and $z$ are $q$-multiplicatively dependent\,---\,a quotient of the form $x^k/z^m$ is an integral power of~$q$
for some integers (not simultaneously zero) $k$ and $m$. Some fundamental issues with the Pad\'e approximations to~\eqref{e01} do not allow to apply the latter
for generic $x$ and $z$ even under the standard setting $p\in\mathbb Z\setminus\{0\}$. Apart from the $q$-logarithm specialization ($x=1$),
some other cases of \eqref{e01} are of historical importance, like
$$
\pi_q=1-4\ell_{p^2}(p,-1)
=1+4\sum_{n=0}^\infty\frac{(-1)^nq^{2n+1}}{1-q^{2n+1}}
=1+4\sum_{n=1}^\infty\frac{q^n}{1+q^{2n}}
=\biggl(\sum_{n=-\infty}^\infty q^{n^2}\biggr)^2.
$$
Several such instances, under the $q$-multiplicative dependence, are studied in \cite{BZ07,BZ08} from arithmetic perspectives, both qualitatively and quantitatively.

As no representation of the function \eqref{e01} through the B\'ezivin-type functions is known, the method from \cite{Be88,Be90} and its generalizations
are not applicable for arithmetic studies of the values of the function. Other available techniques for $q$-hyper\-geo\-metric functions based on the ideas
of the classical Siegel--Shidlovsky method, for example~\cite{AKV01}, fail to apply to the function $\ell_p(x,z)$ for similar reasons.

\section{Principal results}
\label{s2}

In this note we combine the Pad\'e approximation construction from \cite{BZ08} with a version of B\'ezivin's method as developed in \cite{KRVZ09}
to prove the following general result.

\begin{theorem}
\label{th1}
For integer $p$ and nonzero rational $x,z$ such that $|p|>1$, $x\notin\{p,\allowbreak p^2,\allowbreak p^3,\dots\}$ and $|z|<|p|$
\textup(to ensure the convergence of the series in~\eqref{e01}\textup),
the value of $\ell_p(x,z)$ is irrational.
\end{theorem}

A general framework of the approach we use in our proof of the theorem is outlined in \cite{Zu15}.

We now introduce the $q$-hypergeometric function
\begin{gather}
\cF_p(x,y,z)=\sum_{n=0}^\infty\frac{\prod_{k=1}^n(p^k-1)\cdot(-pz)^n}{\prod_{k=1}^n(p^k-x)(p^k-y)}
=\sum_{n=0}^\infty\frac{(-1)^n\prod_{k=1}^n(1-q^k)\cdot q^{n(n-1)/2}z^n}{\prod_{k=1}^n(1-q^kx)(1-q^ky)},
\label{e01b}
\\
|x|<|p| \quad\text{and}\quad |z|<|p|.
\nonumber
\end{gather}
This is an exemplar function to which the known arithmetic methods \cite{AKV01,Be88} are not applicable because of the presence
of the $q$-Pochhammer product $(q;q)_n=\prod_{k=1}^n(1-\nobreak q^k)$ in the numerator of the $n$-th term of the series.
A hypergeometric identity and Theorem~\ref{th1} allow us to establish the following corollary.

\begin{theorem}
\label{th2}
The identity
\begin{equation}
\cF_p(x,z,xz)=\frac{(1-x)(1-z)}{pxz}\,\ell_p(px,pz)
\label{e01c}
\end{equation}
is valid. In particular, for an integer $p\ne0,\pm1$ and rational $x$ and $z$ such that
$0<|x|,|z|<1$, the value of the hypergeometric function $\cF_p(x,z,xz)$ is irrational.
\end{theorem}

In particular, for \eqref{qh1} we obtain
$$
\sum_{n=1}^\infty\frac{q^n}{1-q^n}=\sum_{n=1}^\infty\frac{(-1)^{n-1}q^{n(n+1)/2}}{(1-q^n)\,\prod_{k=1}^n(1-q^k)}
$$
which resembles the counterpart of the $q$-harmonic series in the ``$\frac12$-logarithmic derivative of the Dedekind eta-function'' \cite[Remark on p.~952]{Za01}.

Finally, we remark that the results can be given for non-integer $p=r/s$, $|p|>\nobreak1$, as well under a customary in such situations assumption
$\log|r|>c\log|s|$ for some computable constant $c>0$. Furthermore, irrationality measures for the numbers in question can be produced
showing that they are not Liouvillian.
To make our exposition clear and concise, we do not pursue such quantitative issues here but place a related comment after our proof
of Theorem~\ref{th1}, at the end of Section~\ref{s4}.

\medskip
The standard $q$-notation \cite{GR04} we use below include the (multiple) $q$-Pochhammer symbol
$$
(a_1,a_2,\dots,a_r;q)_n=\prod_{j=1}^n(1-a_1q^{j-1})(1-a_2q^{j-1})\dotsb(1-a_rq^{j-1})
\quad\text{for $n=0,1,2,\dots,\infty$}
$$
(the empty product for $n=0$ is set to be 1), the $q$-binomial coefficients
$$
\qbinom nk=\frac{(q;q)_n}{(q;q)_k\cdot(q;q)_{n-k}}
\quad\text{for $k=0,1,\dots,n$ and $n=0,1,2,\dots$},
$$
and the basic hypergeometric function
$$
{}_r\phi_s\biggl(\begin{matrix} a_1, \, a_2, \, \dots, \, a_r \\ b_1, \, \dots, \, b_s \end{matrix} \biggm|q,z\biggr)
=\sum_{n=0}^\infty\frac{(a_1,a_2,\dots,a_r;q)_n}{(q,b_1,\dots,b_s;q)_n}\bigl((-1)^nq^{n(n-1)/2}\bigr)^{s+1-r}z^n.
$$

\section{Pad\'e approximations}
\label{s3}

In \cite{BZ08}, the following Pad\'e-type approximations to the function \eqref{e01} were constructed:
\begin{align*}
I_n(q,x,z)
&=(xz)^{n+1}\sum_{t=0}^\infty\frac{(q^{t+1};q)_n\cdot q^{(n+1)t}z^t}{(q^{n+1+t}x;q)_{n+1}}
\\
&=(xz)^{n+1}\frac{(q;q)_n}{(q^{n+1}x;q)_{n+1}}\cdot{}_2\phi_1\biggl(\begin{matrix} q^{n+1}, \, q^{n+1}x \\ q^{2n+2}x \end{matrix} \biggm|q,q^{n+1}z\biggr)
\end{align*}
(we choose $m=n$ in the notation of that paper). It is shown in \cite[Lemma 2]{BZ08} that
$$
I_n(q,x,z)=A_n(p)\ell_p(x,z)-B_n(p)-C_n(p),
$$
where $A_n(p)=A_n(p,x)$, $B_n(p)=B_n(p,x)$ and $C_n(p)=C_n(p,x,z)$ are given by the formulae
\begin{align*}
A_n(p)
&=p^{(n+1)(3n+2)/2}x^{-n}\sum_{k=0}^n(-1)^kp^{k(k+1)/2}
\frac{\prod_{j=1}^n(p^{k+j}-x)}{(p;p)_k(p;p)_{n-k}}\,z^{-k},
\displaybreak[2]\\
B_n(p)
&=p^{(n+1)(3n+2)/2}x^{-n+1}\sum_{k=0}^n(-1)^kp^{k(k+1)/2}
\frac{\prod_{j=1}^n(p^{k+j}-x)}{(p;p)_k(p;p)_{n-k}}
\sum_{l=1}^k\frac{z^{l-k}}{p^l-x},
\displaybreak[2]\\
C_n(p)
&=p^{(n+1)(3n+2)/2}z
\sum_{l=0}^{n-1}\frac{x^{-l}}{p^{n-l}-z}
\sum_{k=0}^l(-1)^k\pbinom nk\pbinom{n+l-k}np^{(n-k)(n-k+1)/2}x^k.
\end{align*}
Denote
$$
D_n(p)=D_n(p,x,z)=(xz)^n\frac{\prod_{j=1}^n(p^j-1)(p^j-x)(p^j-z)}{p^{3n(n+1)/2}}
=(xz)^n(q,qx,qz;q)_n.
$$
It follows that
\begin{gather*}
\wt A_n(p)=D_n(p)A_n(p)\in\mathbb Z[p,x,z], \quad
\wt B_n(p)=D_n(p)B_n(p)\in\mathbb Z[p,x,z] \\
\text{and}\quad
\wt C_n(p)=D_n(p)C_n(p)\in\mathbb Z[p,x,z],
\end{gather*}
with the degrees of the polynomials in $x$ and $z$ not exceeding $2n$.
We now define
$$
v_n(\mu)=v_n(\mu;p,x,z)=\wt A_n(p)\mu-\wt B_n(p)-\wt C_n(p),
$$
a linear form in $\mu$ with coefficients from $\mathbb Z[p,x,z]$, and summarise the above as follows:
$$
v_n(\ell_p(x,z))=(xz)^{2n+1}\frac{(q,q,qx,qz;q)_n}{(q^{n+1}x;q)_{n+1}}
\cdot{}_2\phi_1\biggl(\begin{matrix} q^{n+1}, \, q^{n+1}x \\ q^{2n+2}x \end{matrix} \biggm|q,q^{n+1}z\biggr).
$$
Though the symmetry of this expression is not transparent, it follows from one of Heine's classical transformations,
namely, from \cite[Eq.~(III.1)]{GR04}.

\section{Hankel determinants}
\label{s4}

For $n=1,2,\dots$, define the Hankel determinants
$$
V_n(\mu)=V_n(\mu;p,x,z)=\det_{0\le j,l\le n-1}\bigl(v_{j+l}(\mu)\bigr).
$$
These are polynomials in $\mathbb Z[\mu,p,x,z]$ of degree $n$ in $\mu$ and of degree at most $2n(n-1)$ in each of the variables $x$ and $z$.
Choosing
\begin{equation}
v_n^*=(xz)^{-2n-1}v_n(\ell_p(x,z))
=\frac{(q,q,qx,qz;q)_n}{(q^{n+1}x;q)_{n+1}}
\cdot{}_2\phi_1\biggl(\begin{matrix} q^{n+1}, \, q^{n+1}x \\ q^{2n+2}x \end{matrix} \biggm|q,q^{n+1}z\biggr),
\label{vs}
\end{equation}
we obtain
$$
V_n(\ell_p(x,z))
=\det_{0\le j,l\le n-1}\bigl((xz)^{2j+2l+1}v_{j+l}^*\bigr)
=(xz)^{n(2n-1)}\det_{0\le j,l\le n-1}(v_{j+l}^*)
=(xz)^{n(2n-1)}V_n^*.
$$
It is immediate from~\eqref{vs} that the $q$-expansion of $v_n^*$ starts from the constant term, $v_n^*=1+O(q)$, so that
$v_n^*=|p|^{O(n)}$ as $n\to\infty$ (see also \cite[Lemma~5]{BZ08}). Our crucial observation here is that the $q$-expansion
of $V_n^*$ starts from at least $q^{n(n-1)(2n-1)/6}$, so that $|V_n^*|\le|q|^{n^3/3}\exp(O(n^2))$.
In order to see that we introduce the backward shift operator $N\colon v_n\mapsto v_{n-1}$
on the indices, its zeroth power $I\colon v_n\mapsto v_n$\,---\,the identity operator, and define
\begin{equation}
D_l=(N;q)_l=\prod_{j=0}^{l-1}(I-q^jN)=\sum_{k=0}^l(-1)^kq^{k(k-1)/2}\qbinom lkN^k.
\label{Dl}
\end{equation}

\begin{lemma}
\label{lem1}
Let $H(q)=1+\sum_{r=1}^\infty a_rq^r$ be a formal power series starting from $1$ and $t$ a nonnegative integer.
Define
$$
w_n=w_n(H,t)=q^{(n+1)t}\prod_{k=1}^nH(q^k).
$$
Then for $n\ge l$ the $q$-expansion of $D_lw_n$ starts at least from $nl-l(l-1)/2$.
\end{lemma}

\begin{proof}
We extend the action of our operators to the linear spaces
$$
\cW_n=\bigoplus_{m=0}^\infty\{P_m(q)w_n(H,m):P_m(q)\in\mathbb C[q]\}\subset\mathbb C[[q]]
$$
with the goal to demonstrate, by induction on $l$, that the image of an element from $\cW_n$ under $D_l$ belongs to $q^{nl-l(l-1)/2}\cW_{n-l}$
for $l\ne n$. In particular, this will imply the lemma.

The induction base ($l=0$) is trivial, because $nl-l(l-1)/2=0$ for any $n$ in this case.
Assume that the required statement is established for an $l$,
so that for each integer $t\ge0$ we have
$$
D_lw_n(H,t)=q^{nl-l(l-1)/2}\sum_{m=0}^\infty P_m(q)w_{n-l}(H,m)
$$
with some $P_m(q)=P_{m,n,l,t}(q)\in\mathbb C[q]$. Since $D_{l+1}=(I-q^lN)D_l$, in order to verify the statement for $l+1\le n$
we need to demonstrate that $(I-q^lN)(q^{nl-l(l-1)/2}w_{n-l}(H,m))$ belongs to $q^{n(l+1)-l(l+1)/2}\cW_{n-l-1}$ for any $m\ge0$.
We have
\begin{align*}
&
(I-q^lN)(q^{nl-l(l-1)/2}w_{n-l}(H,m))
\\ &\quad
=q^{nl-l(l-1)/2}\cdot q^{(n-l+1)m}\prod_{k=1}^{n-l}H(q^k)
-q^l\cdot q^{(n-1)l-l(l-1)/2}\cdot q^{(n-l)m}\prod_{k=1}^{n-l-1}H(q^k)
\\ &\quad
=q^{n(l+1)-l(l+1)/2}\cdot(q^mH(q^{n-l})-1)\cdot q^{(n-l)(m-1)}\prod_{k=1}^{n-l-1}H(q^k)
\\ &\quad
=q^{n(l+1)-l(l+1)/2}\biggl((-1+q^m)w_{n-l-1}(H,m-1)+q^m\sum_{r=1}^\infty a_rw_{n-l-1}(H,m+r-1)\biggr),
\end{align*}
with the term $(-1+q^m)w_{n-l-1}(H,m-1)$ absent when $m=0$.
This completes the proof of the inductive step, hence of the whole statement and the lemma.
\end{proof}

We remark that Lemma~\ref{lem1} applies to general hypergeometric terms
$$
\frac{(a_1,\dots,a_r;q)_n}{(b_1,\dots,b_s;q)_n}\,q^{(n+1)t}, \qquad\text{where $t$ is a nonnegative integer},
$$
hence to the (finite or infinite) sums of such terms over $t$, as they belong to the vector space $\cW_n$.
In particular, it applies to
$$
v_n^*=\sum_{t=0}^\infty z^t\frac{(q,qx,qz,q^{t+1};q)_n\cdot q^{(n+1)t}}{(q^{n+1+t}x;q)_{n+1}},
$$
so that $q$-expansion of $D_lv_n^*$ starts at least from $q^{nl-l(l-1)/2}$ for any $l\le n$.
Acting on the $j$-th row of the matrix $(v_{j+l}^*)_{0\le j,l\le n-1}$ by $D_j$, which results in the elementary
manipulations of the rows in accordance with the definition \eqref{Dl} of $D_j$, we get the newer matrix $(a_{j,l})_{0\le j,l\le n-1}$ of the same
determinant $V_n^*$, whose entries in the $j$-th row have $q$-order at least $nj-j(j-1)/2$. Thus,
$$
\ord_qV_n^*\ge\sum_{j=0}^{n-1}\biggl((n-1)j-\frac{j(j-1)}2\biggr)=\frac{n(n-1)(2n-1)}6.
$$

The details of the derivation of the estimate
$$
|V_n^*|\le|q|^{n^3/3}\exp(O(n^2)) \qquad\text{as}\quad n\to\infty
$$
from $\ord_qV_n^*\ge n^3/3+O(n^2)$ can be borrowed from the proofs of Lemma 2 and Proposition 3 in \cite{KRVZ09}.
Therefore,
\begin{equation}
|V_n(\ell_p(x,z))|\le|p|^{-n^3/3}\exp(O(n^2)) \qquad\text{as}\quad n\to\infty.
\label{Vn}
\end{equation}

\begin{proof}[Proof of Theorem~\textup{\ref{th1}}]
If $\ell_p(x,z)$ were rational then $(x_0z_0)^{2n(n-1)}V_n(\ell_p(x,z))$ would be an integer for all $n\ge0$, where $x_0$ and $z_0$
denote the denominators of $x$ and~$z$. It follows from Kronecker's criterion \cite[Division~7, Problem~24]{PS76} and the fact that
the series
$$
\sum_{n=0}^\infty v_n(\ell_p(x,z))y^n
=xz\sum_{t=0}^\infty(qz)^t\sum_{n=0}^\infty\frac{(q,qx,qz,q^{t+1};q)_n\cdot(q^tx^2z^2y)^n}{(q^{n+1+t}x;q)_{n+1}}
$$
is not rational (as the interior basic hypergeometric series are not), that this integer is different from zero for infinitely many~$n$, so that
$$
|(x_0z_0)^{2n(n-1)}V_n(\ell_p(x,z))|\ge1
$$
for such $n$. The resulting estimate contradicts the bound in~\eqref{Vn} for sufficiently large $n$, hence $\ell_p(x,z)$ must be irrational.
\end{proof}

Another proof of Theorem~\ref{th1} can be given, without the reference to the Hankel determinants and Kronecker's criterion.
It follows from the definition \eqref{Dl} and Lemma~\ref{lem1} that
\begin{align*}
v_n'(\mu)
&=(xz)^{2n+1}p^{n(n-1)/2}\sum_{k=0}^n(-1)^kp^{-k(k-1)/2-k(n-k)}\pbinom nk(xz)^{k-2n-1}v_{2n-k}(\mu)
\\
&=\sum_{k=0}^n(-1)^kp^{(n-k)(n-k-1)/2}\pbinom nk(xz)^kv_{2n-k}(\mu)\in\mathbb Z[p,x,z]\mu+\mathbb Z[p,x,z]
\end{align*}
for all $n$ and $v_n'(\ell_p(x,z))=(xz)^{2n+1}p^{n(n-1)/2}D_nv_{2n}^*$ satisfies $\ord_qv_n'\ge n(n+1)$, hence is bounded:
$$
|v_n'(\ell_p(x,z))|\le|p|^{-n^2}\exp(o(n^2)).
$$
Assuming the nonvanishing $v_n'(\ell_p(x,z))\ne0$ is proven for a sufficiently dense
sequence of indices $n$, the argument allows not only to demonstrate the irrationality of $\ell_p(x,z)$ but also
to show that it is not a Liouville number by giving the related estimate for its irrationality exponent.

\section{Hypergeometric identity}
\label{shyp}

Before going into details of the proof of identity \eqref{e01c} in Theorem~\ref{th2},
we would like to mention that the inspiration for it comes from the paper \cite{Ma08} by T.~Matala-aho.
A key identity there is
\begin{equation}
(1-t)\sum_{n=0}^\infty\frac{t^n}{(qb;q)_n}
=\sum_{n=0}^\infty\frac{q^{n^2}(bt)^n}{(qb,qt;q)_n},
\qquad |q|<1, \quad |t|<1;
\label{E01}
\end{equation}
he then studies in \cite{Ma08} arithmetic properties of the series
on the right-hand side of~\eqref{E01} by applying Pad\'e-type
techniques to the left-hand side. Identity~\eqref{E01} is classical and established in \cite[Eq.~(12.3)]{Fi88}.
A natural generalization of~\eqref{E01}, namely the form \eqref{E04} from the proof below,
was kindly communicated to us by C.~Krattenthaler~\cite{Kr06}.

\begin{proof}[Proof of Theorem~\textup{\ref{th2}}]
Let us start with \cite[Eq.~(III.9)]{GR04},
$$
{}_3\phi_2\biggl(\begin{matrix} a, \, b, \, c \\ d, \, e \end{matrix} \biggm|q,\frac{de}{abc}\biggr)
=\frac{(e/a,de/(bc);q)_\infty}{(e,de/(abc);q)_\infty}
\cdot{}_3\phi_2\biggl(\begin{matrix} a, \, d/b, \, d/c \\ d, \, de/(bc) \end{matrix} \biggm|q,\frac ea\biggr).
$$
Letting $b$ tend to $\infty$ we obtain
\begin{equation}
{}_2\phi_2\biggl(\begin{matrix} a, \, c \\ d, \, e \end{matrix}
\biggm|q,\frac{de}{ac}\biggr)
=\frac{(e/a;q)_\infty}{(e;q)_\infty}
\cdot{}_2\phi_1\biggl(\begin{matrix} a, \, d/c \\ d \end{matrix} \biggm|q,\frac ea\biggr).
\label{E03}
\end{equation}
Further, letting $c$ tend to $\infty$ we deduce
\begin{equation}
{}_1\phi_2\biggl(\begin{matrix}
a \\ d, \, e \end{matrix}
\biggm|q,\frac{de}a\biggr)
=\frac{(e/a;q)_\infty}{(e;q)_\infty}
\cdot{}_2\phi_1\biggl(\begin{matrix}
a, \, 0 \\ d \end{matrix}
\biggm|q,\frac ea\biggr).
\label{E04}
\end{equation}
The particular case $a=q$, $d=bq$ and $e=tq$ of~\eqref{E04} is exactly identity~\eqref{E01}.

Setting $x=d/c$, $z=e/a$ in the intermediate equation~\eqref{E03} we may write it as follows:
$$
\frac{(z;q)_\infty}{(az;q)_\infty}
\cdot{}_2\phi_1\biggl(\begin{matrix} a, \, x \\ cx \end{matrix} \biggm|q,z\biggr)
={}_2\phi_2\biggl(\begin{matrix} a, \, c \\ cx, \, az \end{matrix} \biggm|q,xz\biggr).
$$
Then taking $a=c=q$ we obtain
\begin{align*}
(1-x)\sum_{n=0}^\infty\frac{z^n}{1-q^nx}
&={}_2\phi_1\biggl(\begin{matrix} q, \, x \\ qx \end{matrix} \biggm|q,z\biggr)
=\frac{(qz;q)_\infty}{(z;q)_\infty}
\cdot{}_2\phi_2\biggl(\begin{matrix} q, \, q \\ qx, \, qz \end{matrix} \biggm|q,xz\biggr)
\\
&=\frac1{1-z}\sum_{n=0}^\infty
\frac{(q;q)_n}{(qx,qz;q)_n}\,q^{n(n-1)/2}(-xz)^n.
\end{align*}
Finally, performing simple manipulations we cast the resulting identity as \eqref{e01c}. This concludes our proof of the theorem.
\end{proof}

\section{Related questions}
\label{sfin}

As already pointed out, there are no general methods at present to prove the irrationality of the values of $q$-hypergeometric series
like \eqref{e01b}, when $q$-Pochhammer products appear in the numerators of the terms, for a sufficiently generic set of parameters.

A different from~\eqref{e01b} three-parameter generalization of $\ell_p(x,z)$ can be considered,
which is suggested by rational Pad\'e-type approximations in \cite{BV05,BV09,SvA09,Zu02a,Zu06} that generalize
the approximations to the $q$-logarithm function. Namely, we introduce the function
\begin{equation}
\cL_p(x,y,z)
=\sum_{n=1}^\infty\frac{p^nz^n}{(p^n-x)(p^n-y)}
=\begin{cases}
\dfrac{\ell_p(x,z)-\ell_p(y,z)}{x-y} &\text{if $x\ne y$}, \\[1.5mm]
\dfrac{\d}{\d x}\ell_p(x,z) &\text{if $x=y$};
\end{cases}
\label{e02}
\end{equation}
where as in \eqref{e01} the variables $x$, $y$, and $z$ live inside the circle
of radius $|p|$ centred at the origin. The function~\eqref{e02}
satisfies the obvious symmetry relation $\cL_p(x,y,z)=\cL_p(y,x,z)$ and also the functional equation
\begin{equation*}
(x-y)\cL_p(x,y,z)
+(y-z)\cL_p(y,z,x)
+(z-x)\cL_p(z,x,y)
=0,
\end{equation*}
while the $x=y$ specialization
\begin{equation}
\cL_p(x,x,z)
=\sum_{n=1}^\infty\frac{p^nz^n}{(p^n-x)^2}
=\frac zx\sum_{n=1}^\infty\frac{nx^n}{p^n-z}
\label{e03}
\end{equation}
contains the $q$-dilogarithmic function~\cite{Zu06} as the case $z=1$.
In particular, the first two $q$-zeta values \cite{Zu02b} are given by
$$
\zeta_q(1)=\ell_p(1,1)
\quad\text{and}\quad
\zeta_q(2)=\cL_p(1,1,1).
$$
Some further arithmetically interesting specializations include
\begin{gather*}
\cL_p(1,-1,1)=\sum_{n=1}^\infty\frac{p^n}{p^{2n}-1},
\quad
\cL_p(i,-i,1)=\sum_{n=1}^\infty\frac{p^n}{p^{2n}+1}=\frac{\pi_q-1}4,
\\
\cL_p(e^{2\pi i/3},e^{-2\pi i/3},1)=\sum_{\nu=1}^\infty\frac{p^n}{p^{2n}+p^n+1}
\\ \intertext{as well as}
\cL_p(-1,-1,1)=\cL_p(1,1,1)-4\cL_{p^2}(1,1,1)=\zeta_q(2)-4\zeta_{q^2}(2).
\end{gather*}

Another generalization of the function~\eqref{e01b} is motivated by the fact that it is a ``half'' of Kronecker's famous identity
\begin{gather*}
\frac1{qx}\,\ell_p(x,z/q)-z\ell_p(1/(qx),1/z)
=\sum_{n\in\mathbb Z}\frac{z^n}{1-q^nx}
=\frac{(q,q,xz,q/(xz);q)_\infty}{(x,q/x,z,q/z;q)_\infty},
\\
|q|<|z|<1,
\end{gather*}
and the existence, though much more recent, of an elegant expression for the double Lambert series
\begin{gather*}
\biggl(\sum_{m,n\ge0}-\sum_{m,n<0}\biggr)\frac{q^{mn}y^mz^n}{1-q^{m+n}x}
=\biggl(\sum_{l,m,n\ge0}-\sum_{l,m,n<0}\biggr)q^{lm+mn+nl}x^ly^mz^n,
\\
|q|<|x|<1, \quad |q|<|y|<1, \quad |q|<|z|<1,
\end{gather*}
due to E.~Mortenson~\cite{Mo16}. This brings to consideration the function
\begin{align*}
\Lambda_p(x,y,z)
&=\sum_{l,m,n>0}q^{lm+mn+nl}x^ly^mz^n
=x\sum_{m,n>0}\frac{(qy)^m(qz)^n}{1-q^{m+n}x}
\\
&=x\sum_{m,n>0}\frac{y^mz^n}{p^{m+n}-x},
\qquad |y|<|p|, \quad |z|<|p|.
\end{align*}

We leave a hope that arithmetic studies of the values of the functions may use the techniques of the present note.

\medskip
\noindent
\textbf{Acknowledgements.}
I would like to thank Christian Krattenthaler for his earlier communication~\cite{Kr06}, Eric Mortenson for sharing his
remarkable identity from~\cite{Mo16} with me as well as Igor Rochev and Armin Straub for useful comments.


\end{document}